\documentclass[ECP]{ejpecp} 
\graphicspath{ {Images/} }

\usepackage{subfigure}
\usepackage{tikz}
\newcommand{\cc}{c}

\SHORTTITLE{Coexistence in random growth} 

\TITLE{Coexistence in a random growth model with competition} 

\AUTHORS{%
  Shane~Turnbull\footnote{Department of Mathematics and Statistics, Lancaster University, United Kingdom.
    \EMAIL{s.m.turnbull@lancaster.ac.uk, a.g.turner@lancaster.ac.uk}}
  \and 
  Amanda Turner\footnotemark[1]}


\AMSSUBJ{60Fxx; 60K35} 

\SUBMITTED{July 7, 2019} 
\ACCEPTED{February 28, 2020} 

\VOLUME{25}
\YEAR{2020}
\PAPERNUM{26}
\DOI{https://doi.org/10.1214/20-ECP304}

\newcommand{\br}{\begin{color}{red}}
\newcommand{\er}{\end{color}}

\ABSTRACT{We consider a variation of the Hastings-Levitov model HL(0) for random growth in which the growing cluster consists of two competing regions. We allow the size of successive particles to depend both on the region in which the particle is attached, and the harmonic measure carried by that region. We identify conditions under which one can ensure coexistence of both regions. In particular, we consider whether it is possible for the process giving the relative harmonic measures of the regions to converge to a non-trivial ergodic limit.
} 

\begin{document}



\section{Introduction}
We consider planar random growth models in which clusters grow by the successive attachment of single particles. In the specific class of models that we study, such clusters are encoded as compositions of conformal mappings. The simplest model of this type is the HL(0) model, proposed by Hastings and Levitov \cite{HAS98}, in which clusters are constructed as successive compositions of i.i.d.~mappings. This model has been well studied (see \cite{NOR12, SIL17} amongst others). In physical models for random growth, specifically Laplacian random growth models, the growth rate along the cluster boundary depends on the harmonic measure of the cluster boundary. This dependency makes the analysis considerably less tractable. In this paper, we introduce dependency on harmonic measure into a variant of the HL(0) model, through competition. We define a random growth model in which the cluster is made up of two competing regions and incoming particles are added to the region to which they attach. Dependency on harmonic measure is introduced by allowing the growth of each competing region to depend on the relative harmonic measure of that region. We explore whether it is possible for both regions to coexist indefinitely (in the sense that there is a positive probability that each region has positive harmonic measure for all time), or whether it is always the case that one region will dominate to the exclusion of the other.

\subsection{Conformal models for random growth}
The idea of using conformal mappings to represent random growth in two-dimensions has been around since the work of Hastings and Levitov \cite{HAS98}. The primary benefit of this approach is that it provides a purely analytic, rather than geometric, representation of a randomly growing cluster which enables one to exploit analytic techniques. In this section we provide the general framework into which our models fall.

For $\cc>0$, let $f_\cc$ be the unique conformal bijection
\begin{equation*}
f_\cc: \Delta:=\{z \in \mathbb{C}:|z|>1 \} \cup \{ \infty \} \rightarrow D:=\Delta \setminus (1,1+d]
\end{equation*} 
with $f_\cc(z)=e^\cc z+ \mathcal{O}(1)$ at infinity, where $d=d(\cc)$ and $\cc$ are related via the equation
\(
e^\cc=1+d^2/(4(1+d)).
\)
Observe that $d \asymp \cc^{1/2}$ as $\cc \rightarrow 0$. This mapping represents attaching a particle (a `slit' of length $d$, or equivalently of logarithmic capacity $\cc$) to the unit circle $\mathbb{T}$ at the point 1. For $\theta \in [-1,1)$, the mapping
\begin{equation*}
f_\cc^\theta(z)=e^{\pi i \theta} f_\cc (e^{-\pi i \theta}z)
\end{equation*}
represents attaching a particle with logarithmic capacity $c$ at position $e^{\pi i \theta}$ on $\mathbb{T}$.

Now consider a sequence of positions $(\theta_n)_{n \in \mathbb{N}}$ in $[-1,1)$, a sequence of logarithmic capacities $(c_n)_{n \in \mathbb{N}}$ in $(0,\infty)$, and a sequence of times $(t_n)_{n \in \mathbb{N}_0}$ with $0=t_0 < t_1 < t_2 < \cdots$. Define
\[
\Phi_t(z) = 
\begin{cases}
z \quad &\mbox{ if } \quad t_0 \leq t < t_1; \\
f_{\cc_1}^{\theta_1} \circ \cdots \circ f_{\cc_n}^{\theta_n}(z) \quad &\mbox{ if } \quad t_n \leq t<t_{n+1}, \ n \geq 1.
\end{cases} 
\]
Then $(\Phi_t)_{t \geq 0}$ is a cadlag process of conformal mappings, each of which maps the exterior unit disk to the complement of a compact set. In other words, 
\begin{equation*}
\Phi_t:\Delta \rightarrow \mathbb{C} \setminus K_t.
\end{equation*}
The sets $(K_t)_{t \geq 0}$ are called clusters which satisfy $K_{s} \subseteq K_t$ for $s \leq t$. If $t_n \leq t < t_{n+1}$, then the set $K_t$ represents the growing cluster after the addition of $n$ particles, and  $K_{t_{n+1}} = K_{t_n} \cup P_{n+1}$ where 
\begin{equation} \label{eq:pdistort}
P_{n+1}=\{ \Phi_{t_n}(\lambda e^{i \pi \theta_{n+1}}) : \lambda \in (1, 1+d(c_{n+1})] \}.
\end{equation}

\begin{figure}[h]
\begin{center}
\includegraphics[scale=0.48]{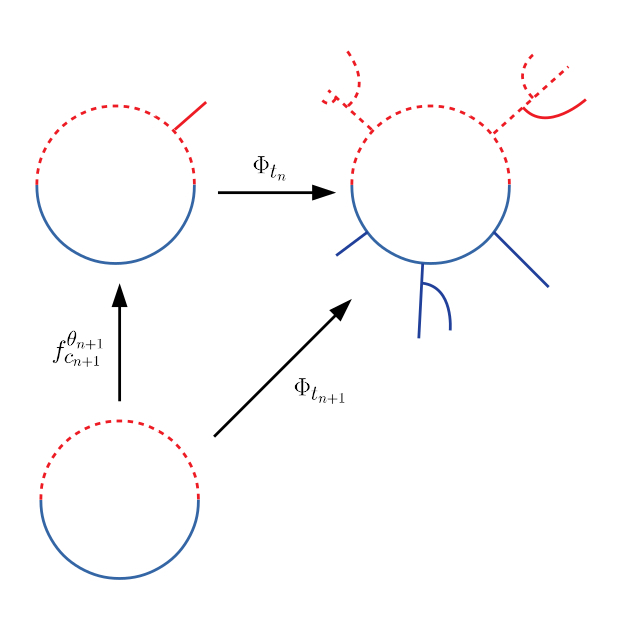}
\caption{Diagram illustrating how growing clusters can be represented as compositions of conformal mappings.}
\label{fig:model}
\end{center}
\end{figure}

By choosing the sequences $(\theta_n)_{n \in \mathbb{N}}$, $(c_n)_{n \in \mathbb{N}}$ and $(t_n)_{n \in \mathbb{N}_0}$ in different ways, one obtains a wide class of growth processes. The HL(0) process mentioned above is obtained by taking $(\theta_n)_{n \in \mathbb{N}}$ to be i.i.d~uniform $[-1,1)$ random variables, $c_n=c$  and $t_n=n$
for all $n \in \mathbb{N}$. A continuous-time embedding of HL(0) is obtained by taking $(\theta_n)_{n \in \mathbb{N}}$ and $(c_n)_{n \in \mathbb{N}}$ as before, but letting $(t_n)_{n \in \mathbb{N}_0}$ be a constant-rate Poisson process. For other choices in the literature, see \cite{STV} and the references therein. 

In this paper, we introduce two competing regions by colouring the upper half of the unit circle red and the lower half blue. Arriving particles then take the colour of the region that they are attached to as illustrated in Figure \ref{fig:model}. In Section \ref{intro:comp} we explain how we choose the sequences $(\theta_n)_{n \in \mathbb{N}}$, $(c_n)_{n \in \mathbb{N}}$ and $(t_n)_{n \in \mathbb{N}_0}$ for our specific model.

\subsection{Harmonic measure} \label{sec:harm}

The motivation behind the Hastings-Levitov model was to model growing clusters formed by the aggregation of diffusing particles. In particular, the aim was to model a process known as diffusion limited aggregation (DLA) \cite{WIT83}. In the DLA model, particles are released one by one from `infinity' and follow the trajectory of a Brownian motion until they hit the cluster at which point each particle sticks. This model is very hard to analyse mathematically, and essentially only one rigorous result has been proved about DLA in the almost 40 years since it was first proposed \cite{KES87}. DLA is an example of a Laplacian random growth model in that the rate of growth is determined by harmonic measure on the cluster boundary.

\begin{definition}[Harmonic measure]
Let $D \subset \mathbb{C}$ and let $\partial D$ be the boundary of $D$. Then for any $A \subseteq \partial D$ and $x \in D$, the harmonic measure of the set $A$ as seen from $x$ is defined to be 
\begin{equation*}
\mu^x_D(A)=\mathbb{P}(B_{\tau} \in A | B_0=x),
\end{equation*}
where $(B_t)_{t \geq 0}$ is complex Brownian motion and $\tau=\inf_{t>0} \{ B_t \in \partial D\}$ is the first exit time from $D$.
\end{definition}

In DLA, the attachment position of successive particles is given by the distribution of harmonic measure as seen from infinity, where the point at infinity is defined via the Riemann sphere. Therefore, using the notation of the previous section, if $K_{t_n}$ is a DLA cluster with $n$ particles, corresponding to conformal map $\Phi_{t_n}$, then $\theta_{n+1}$ must be chosen in such a way that $\Phi_{t_n}(e^{\pi i \theta_{n+1}})$ is distributed according to harmonic measure on $\partial K_{t_n}$.

Directly deriving the harmonic measure on $\partial K_{t_n}$ is complicated, since the set $K_{t_n}$ can be quite intricate. However, the conformal mapping construction turns out to be convenient here.
For $a \leq b < a+2$, let $I_{a,b} \subseteq \mathbb{T}$ be the arc of the unit circle between $e^{ \pi ia}$ and $e^{ \pi ib}$, taken in an anticlockwise direction.
As the distribution of Brownian motion is invariant under conformal mappings, 
\[
\mu^{\infty}_{K_{t_n}}(\Phi_{t_n}(I_{a,b})) = \mu^{\infty}_{\Delta}(I_{a,b}) = (b-a)/2.
\]
Hence, in order to model DLA as a conformal mapping model, each $\theta_n$ should be chosen uniformly on $[-1, 1)$, as for HL(0). This connection provided the original motivation for defining the Hastings-Levitov family of models. (The HL(0) model is not proposed as a model for DLA, however, as the successive composition of conformal maps distorts the size and shape of each added particle, as in \eqref{eq:pdistort}. This distortion can be corrected for, by allowing the capacity sequence $(c_n)_{n \in \mathbb{N}}$ to depend on the harmonic measure. Further detail is given in \cite{HAS98}. So far the Hastings-Levitov version of DLA has proved intractable to mathematical analysis.) 

Define $\gamma_{\cc}^{\theta}: \mathbb{R} \to \mathbb{R}$ by
\[
\gamma_{\cc}^{\theta}(x)= \frac{1}{\pi i}\log \left( (f_\cc^{\theta})^{-1}(e^{\pi i x}) \right)
\]
and let $\gamma_{\cc} = \gamma_{\cc}^{0}$,
where the branch of the logarithm is chosen so that $\gamma_\cc$ maps $(-1,1)$ into itself and $\gamma_{\cc}(x+2n) = 2n+ \gamma_{\cc}(x)$ for all $n \in \mathbb{Z}$. Then \begin{equation*}
\gamma_\cc^\theta(x)=\theta+\gamma_\cc(x-\theta).
\end{equation*}
By direct computation it can be shown that, for $0 < |x| < 1$, 
\begin{equation} \label{eq:gammac}
\gamma_\cc(x)=2\pi^{-1}\mbox{sgn}(x)\tan^{-1}\sqrt{e^\cc \tan^2(\pi x/2)+e^\cc-1}.
\end{equation} 
The map $\gamma_\cc$ has a discontinuity at $0$ (and hence at every even integer point). The value that $\gamma_\cc$ takes at this point will turn out not to matter, so without loss of generality we may assume that $\gamma_\cc$ is right-continuous. 
Set
\[
Z_t(x) = 
\begin{cases}
x \quad &\mbox{ if } \quad t_0 \leq t < t_1; \\
\gamma_{\cc_n}^{\theta_n} \circ \cdots \circ \gamma_{\cc_1}^{\theta_1}(x) \quad &\mbox{ if } \quad t_n \leq t<t_{n+1}, \ n \geq 1.
\end{cases} 
\]

The sequence $(Z_t)_{t \geq 0}$ describes the evolution of harmonic measure on the cluster boundary in that if $a \leq b < a + 2$ and $A_t$ is the section of $\partial K_t$ which lies between $e^{ \pi ia}$ and $e^{ \pi ib}$, taken anticlockwise, then
\[
\mu_t(a,b):=\mu_{K_t}^{\infty}(A_t) = (Z_t(b)-Z_t(a))/2. 
\]

In \cite{NOR12} it is shown that, under an appropriate scaling, the evolution of harmonic measure on the boundary of the HL(0) cluster converges to the Brownian web. Specifically, if $(\theta_n)_{n \in \mathbb{N}}$ are i.i.d.~uniform $[-1,1)$ random variables, $c_n=\cc$ for all $n \in \mathbb{N}$ and $(t_n)_{n \in \mathbb{N}_0}$ is a Poisson process with rate of order $\Theta( \cc^{-3/2})$, then for all $0 \leq x_1 < \cdots < x_k < 2$, $(Z_t(x_1), \dots, Z_t(x_k))_{t \geq 0}$ converges in distribution to $(B_t(x_1), \dots, B_t(x_k))_{t \geq 0}$ as $\cc \to 0$, where $(B_t(x_1), \dots, B_t(x_k))_{t \geq 0}$ is a family of coalescing Brownian motions on the circle starting from $(x_1, \dots, x_k)$. In particular, as any two Brownian motions on the circle eventually coalesce, this implies that for all $a < b < a + 2$, $\mu_t(a,b)$ is eventually either 0 or 1. As the attachment position of particles is distributed according to harmonic measure, this means that it is not possible for infinitely many particles to attach both between $e^{\pi ia}$ and $e^{\pi ib}$, and between $e^{\pi ib}$ and $e^{\pi i(a+2)}$.

\subsection{Introducing competition}
\label{intro:comp}
The HL(0) model is the simplest conformal model for random growth to analyse, as the individual mappings which are composed to make up the cluster are i.i.d. In the Laplacian models corresponding to physical growth, the growth rate of the cluster depends non-trivially on the harmonic measure which makes the analysis considerably less tractable. In this section, we define a variant of the HL(0) model, in which dependency on harmonic measure is introduced through competition. 

As before, in order to grow a random cluster, we require a sequence of angles $(\theta_n)_{n \in \mathbb{N}}$ in $[-1,1)$, a sequence of capacities $(c_n)_{n \in \mathbb{N}}$ in $(0,\infty)$, and a sequence of times $(t_n)_{n \in \mathbb{N}_0}$ with $0=t_0 < t_1 < t_2 < \cdots$. We aim to establish results in the small-particle limit. That is, we let $\cc>0$ be a scaling parameter which controls the logarithmic capacities of our particles so that $c_n = \Theta(c)$ for all $n \in \mathbb{N}$, and we consider the limiting behaviour of the cluster as $c \to 0$ (where the rate of arrivals $(t_n)_{n \in \mathbb{N}_0}$ is tuned appropriately to produce a non-trivial limit).   
As in the continuous-time embedding of HL(0), we take the angles $(\theta_n)_{n \in \mathbb{N}}$ to be i.i.d.~uniform in $[-1,1)$ and the arrival times $(t_n)_{n \in \mathbb{N}_0}$ to be a Poisson process with constant rate $r(c)$, which will be tuned to provide a non-trivial limit as $c \to 0$. However, we now introduce two competing regions, blue and red, as follows. At time 0, we split the unit disk $K_0$ into two regions by colouring the upper half red and the lower half blue. When each subsequent particle arrives, it takes the colour of the region that it is attached to. We also define the stochastic process $X_t^c$ to be twice the harmonic measure of the red region at time $t$. Note that $X_0^c = 2\mu_0(0, 1) = 1$, and in general 
\[
X_t^c=2\mu_t(0,1)=Z_t(1)-Z_t(0).
\] 
Figure \ref{fig:1} illustrates how the harmonic measure of the red region changes due to the arrival of a particle. 

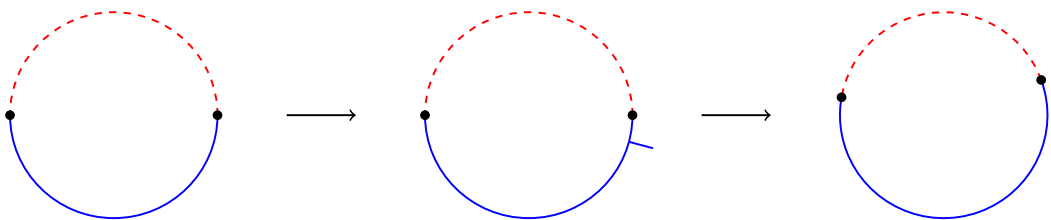
\begin{figure}[h]
\begin{center}
\scalebox{0.91}{
\begin{tikzpicture}
\draw[red,thick,dashed] (-6,0) arc (0:180:1.5cm);
\draw[blue,thick] (-6,0) arc (0:-180:1.5cm);
\fill (-6,0)  circle[radius=2pt];
\fill (-9,0)  circle[radius=2pt];

\draw[black,thick,->] (-5,0) -- (-4,0) ;

\draw[red,thick,dashed] (0,0) arc (0:180:1.5cm);
\draw[blue,thick] (0,0) arc (0:-180:1.5cm);
\fill (0,0)  circle[radius=2pt];
\fill (-3,0)  circle[radius=2pt];
\draw[blue,thick] (-0.05111126056,-0.38822856765) -- +(0.35,-0.09378221734);

\draw[black,thick,->] (1,0) -- (2,0) ;

\draw[red,thick,dashed] (5.90953893118,0.51303021498) arc (20:170:1.5cm);
\draw[blue,thick] (5.90953893118,0.51303021498) arc (20:-190:1.5cm);
\fill (5.90953893118,0.51303021498)  circle[radius=2pt];
\fill (3.02278837048,0.2604722665)  circle[radius=2pt];

\end{tikzpicture}}
\caption{In this illustrative example, the image on the left shows the initial system with the black dots indicating the boundary between the red and blue regions. The central image shows a particle joining and acquiring the colour of the region in which it has landed. The image on the right shows the effect of applying the conformal mapping $(f_{c_1}^{\theta_1})^{-1}$, which absorbs the particle into the boundary. This causes the black boundary points to move, but in such a way that the relative harmonic measure of the coloured regions is preserved between the second and third image. The process $X_t^c$ records how the harmonic measure of the red region evolves beween the first and third images.} \label{fig:1}

\end{center}
\end{figure}

We introduce dependency on harmonic measure into the system by allowing the distribution of $c_{n+1}$ to depend on $X^c_{t_n}$. To do this, we introduce functions $s^\pm(x,c): (0,2) \times (0, \infty) \to (0,\infty)$. Conditional on $X^c_{t_n}=x$, the logarithmic capacity $c_{n+1}$ takes the value $cs^+(x,c)$ if the particle lands in the red region (i.e.~if $e^{\pi i \theta_{n+1}} \in \Phi_{t_n}(I_{0, 1})$, which is an event of probability $x/2$) and $c_{n+1}$ takes the value $cs^-(x,c)$ otherwise. To ensure the correct scaling, we assume that $s^+(x,c) \to s^+(x)$ and $s^-(x,c) \to s^-(x)$, uniformly on $x \in (0,2)$ as $c \to 0$, where $s^+$ and $s^-$ are Lipschitz continuous functions on $(0,2)$, not both identically zero.

In the HL(0) model, a consequence of the evolution of harmonic measure converging to the Brownian web is that the process $(X_t^c)_{t \geq 0}$ converges in distribution to a (rescaled) Brownian motion, stopped on hitting 0 or 2.  This means that one of the two regions will dominate in the sense that eventually all arriving particles will take the same colour. The principal question that we wish to explore is whether it is possible to chose the functions $s^\pm(x,c)$ in order to ensure the coexistence of the red and blue regions. We are particularly interested in whether these functions can be chosen to ensure that the harmonic measure of each area stabilises to a non-trivial ergodic process. 

Our main theorem is that this is possible, but only if the functions $s^{\pm}(x,c)$ are chosen very carefully. The precise result is given as Theorem \ref{thrm:mainerg}, but essentially states the following.

\begin{theorem}
Suppose that 
\(
s^+(x)=s^-(x)=s(x)>0
\)
 for all $x \in (0,2)$, that
 \(
 r(c) =  c^{-3/2}
 \)
 and that 
\[
c^{-1/2} \log (c^{-1}) (s^+(x,c)-s^-(x,c)) \rightarrow h(x) 
\]
uniformly in $x$ as $c \to 0$ for some Lipschitz continuous function $h$ on $(0,2)$ for which $h(x)/s(x)^{3/2}$ is locally integrable on $(0,2)$. Set
\[
m(x)=\frac{3 \pi^3}{32 s(x)^{3/2}} \exp \left (  \int_1^x \frac{3 \pi h(y)}{16 s(y)^{3/2}} dy \right ).
\]
If $s^{\pm}(x,c)$ are chosen so that $s$ and $h$ satisfy the further assumption that $m(x)$ is integrable on $(0,2)$ (and an additional technical condition), then $X_t^c$
converges in distribution as $t \to \infty$ and the limit converges as $c \to 0$ to some continuous random variable which has density proportional to $m(x)$ on $(0,2)$.  
\end{theorem}

One of the challenges in studying true Laplacian random growth models is that the harmonic measure on the cluster boundary is a Markov process in an infinite-dimensional space. By contrast, in the model defined above, the process $(X_t^c)_{t \geq 0}$ giving (twice) the harmonic measure of the red region is just a real-valued pure-jump Markov process and is therefore amenable to standard techniques for analysing scaling limits of Markov processes. The question of coexistence can then be answered by studying the long-time behaviour of $(X_t^c)_{t \geq 0}$ in the limit as $c \to 0$. This analysis makes up the remainder of the paper. In Section \ref{sec:diff_est} we obtain the scaling limit $X_t$ of the process $X_t^c$ as $c \to 0$, under suitable assumptions on $r(c), s^+(x,c), s^-(x,c)$. In Section \ref{sec:limit} we derive the asymptotic distribution of $X_t$ as $t \to \infty$, and explore conditions under which $\lim_{c \to 0} \lim_{t \to \infty} X_t^c = \lim_{t \to \infty} X_t$. Finally we end with an illustrative example in Section \ref{sec:example}. 

\section{Diffusion estimates} \label{sec:diff_est}

In this section, we obtain the scaling limit process $(X_t)_{t \geq 0}$ to which $(X_t^c)_{t \geq 0}$ converges as $c \to 0$, under suitable assumptions on $r(c), s^+(x,c), s^-(x,c)$. 

Throughout, convergence of stochastic processes will be taken to mean weak convergence in $D[0,\infty)$, the space of cadlag functions on $[0,\infty)$ equipped with the Skorokhod topology.

Let $\Delta X_{t_n}^c = X^c_{t_n} - X^c_{t_{n-1}}$. Then
\begin{align*}
\Delta X_{t_n}^c &= \Delta Z_{t_n}(1) - \Delta Z_{t_n}(0) \\
&= \left ( \gamma_{c_n}^{\theta_n}(Z_{t_{n-1}}(1))- Z_{t_{n-1}}(1) \right ) - \left ( \gamma_{c_n}^{\theta_n}(Z_{t_{n-1}}(0))- Z_{t_{n-1}}(0) \right ) \\
&= \tilde{\gamma}_{c_n}(Z_{t_{n-1}}(1) - \theta_n)- \tilde{\gamma}_{c_n}(Z_{t_{n-1}}(0)-\theta_n),
\end{align*}
where $\tilde{\gamma}_c(x)=\gamma_c(x)-x$. Observe that $\tilde{\gamma}_c$ is asymmetric and periodic with period 2. Due to the rotational symmetry of the model, the distribution of this process is unchanged if $\theta_n$ is taken to uniformly distributed on $[Z_{t_{n-1}}(1)-2,Z_{t_{n-1}}(1))$. Let $\theta_n'= Z_{t_{n-1}}(1)-\theta_n$. Then $\theta_n'$ is a uniform $[0,2)$ random variable. Using that 
\[
c_{n}= \begin{cases} cs^+(X_{t_{n-1}}^c,c) \quad &\mbox{ if } \quad e^{ \pi i \theta_{n}} \in I_{Z_{t_{n-1}}(0), Z_{t_{n-1}}(1)}, \\
cs^-(X_{t_{n-1}}^c,c) &\mbox{ otherwise},
\end{cases} 
\]
we get
\begin{align*}
\Delta X_{t_n}^c &= 
\begin{cases}
\tilde{\gamma}_{cs^+(X_{t_{n-1}}^c,c)}(\theta_n')- \tilde{\gamma}_{cs^+(X_{t_{n-1}}^c,c)}(\theta_n'-X_{t_{n-1}}^c) \quad &\mbox{ if } \theta_n' \in (0,X_{t_{n-1}}^c), \\
\tilde{\gamma}_{cs^-(X_{t_{n-1}}^c,c)}(\theta_n')- \tilde{\gamma}_{cs^-(X_{t_{n-1}}^c,c)}(\theta_n'-X_{t_{n-1}}^c) \quad &\mbox{ if } \theta_n' \in (X_{t_{n-1}}^c,2).
\end{cases}
\end{align*}
Hence the stochastic process $(X_t^c)_{t \geq 0}$ is a pure-jump real-valued Markov process with kernel   
\begin{align}\label{eq:kernel}
\nonumber K^c(x,dy)dt := & \ \mathbb{P}\left(\left . t_n \in dt, \Delta X^c_{t_n} \in dy \right | t_n>t,X^c_{t_{n-1}}=x \right) \\
 = &\ \frac{1}{2} r(c) \int_{0}^x \delta_y \left(\tilde{\gamma}_{cs^+(x,c)}(\theta)-\tilde{\gamma}_{cs^+(x,c)}(\theta-x)\right) d\theta dy dt \\
\nonumber &  \ + \frac{1}{2} r(c) \int_{x}^{2} \delta_y \left(\tilde{\gamma}_{cs^-(x,c)}(\theta)-\tilde{\gamma}_{cs^-(x,c)}(\theta-x)\right) d\theta dy dt,
\end{align}
where $\delta_y$ is the Dirac delta function at $y$.

The scaling limit of this process can be found using Theorem 7.4.1 and Corollary 7.4.2 in \cite{ETH86}, which we restate here for convenience.

\begin{theorem}[Ethier and Kurtz] \label{thrm:kurtz}
Let $a=(a_{i,j})$ be a Lipschitz continuous, symmetric, non-negative definite $d \times d$ matrix valued function on $\mathbb{R}^d$ and let $b: \mathbb{R}^d \longrightarrow \mathbb{R}^d$ be Lipschitz continuous. Let $K^c(x,dy)$ be the kernel associated with the process $(X^c_t)_{t \geq 0}$, which takes values on some subset $I \subseteq \mathbb{R}^d$ and define 
\begin{equation*}
b^c(x)=\int_{\mathbb{R}^d} y K^c(x,dy) \hspace{1cm} \mbox{and} \hspace{1cm} a^c(x)=\int_{\mathbb{R}^d} yy^T K^c(x,dy).
\end{equation*} 
Suppose that,
\begin{equation*}
\sup_{x \in I}|a^c(x)-a(x)| \rightarrow 0 \hspace{1cm} \mbox{and} \hspace{1cm} \sup_{x \in I}|b^c(x)-b(x)| \rightarrow 0
\end{equation*}
and that 
\begin{equation*}
\sup_{t>0}|X^c_t-X^c_{t-}| \rightarrow 0
\end{equation*}
as $c  \rightarrow 0$. If $X_0^c\to X_0$ weakly as $c \to 0$, then $X_t^c \rightarrow X_t$ weakly in $D[0,\infty)$ where $(X_t)_{t \geq 0}$ is a solution to the stochastic differential equation given by 
\begin{equation} \label{eq:diff}
dX_t=b(X_t)dt+\sigma(X_t)dB_t, 
\end{equation}
where $a(x)= \sigma(x)^T\sigma(x)$.
\end{theorem}

\begin{definition}[Generator]
Suppose that $(X_t)_{t \geq 0}$ is a solution of the Markovian SDE
\begin{equation}\label{eq:marksde}
dX_t=b(X_t)dt+\sigma(X_t)dB_t+\int_\mathbb{R} \zeta(X_{t-},u) N(dt,du)
\end{equation}
where $N(dt,du)$ is a Poisson random measure with intensity $\nu(du)dt$. The generator of the process $(X_t)_{t \geq 0}$ is the linear operator $L$ defined by
\[
Lf(x) = b (x) f'(x) + \frac{1}{2} a(x) f''(x) + \int_\mathbb{R} \left ( f(x+\zeta(x,u)) - f(x) \right) \nu(du), 
\]
where $a(x)= \sigma(x)^T\sigma(x)$.
\end{definition}

When the jump-sizes $\zeta(x,u)=0$ for all $x$, the process $(X_t)_{t \geq 0}$ satisfying \eqref{eq:marksde} is just the solution to the SDE \eqref{eq:diff}, whereas when $b(x)=0=\sigma(x)$ for all $x$, $(X_t)_{t \geq 0}$ is a pure-jump process with kernel 
\[
K(x,dy)= \int_{\mathbb{R}} \delta_y(\zeta(x,u)) \nu(du) dy.
\]
Theorem \ref{thrm:kurtz} is therefore a special case of a more general result in \cite{ETH86}: under mild conditions, if the initial distribution and generator of $(X_t^N)_{t \geq 0}$ converge to those of $(X_t)_{t \geq 0}$, then $X_t^N \to X_t$ weakly in $D[0,\infty)$.

By applying Theorem \ref{thrm:kurtz} to the kernel \eqref{eq:kernel} we obtain the following.
 
\begin{proposition}\label{thrm:maindiff}
Suppose there exist Lipschitz continuous functions $a,b: (0,2) \to \mathbb{R}$ such that 
\begin{equation*}
c \log (c^{-1}) r(c) (s^+(x,c)-s^-(x,c))/\pi^2  \rightarrow b(x) 
\end{equation*}
and
\begin{equation*}
\frac{16}{3 \pi^3}c^{3/2}r(c)\left (s^+(x,c)^{3/2}+s^-(x,c)^{3/2} \right )\rightarrow a(x) 
\end{equation*} 
uniformly in $x$ as $c \rightarrow 0$. Let $(X_t)_{t \geq 0}$ be the solution to \eqref{eq:diff} with $X_0=1$. Then $X_t^c \rightarrow X_{t \wedge \tau}$ weakly in $D[0,\infty)$, where $\tau = \inf\{t>0: X_t \notin (0,2) \}$. 
\end{proposition}
\begin{proof}
We compute the functions $a^c(x)$ and $b^c(x)$. First,
\begin{align*}
b^c(x)
=& \frac{1}{2} r(c) \int_{0}^x  \left(\tilde{\gamma}_{cs^+(x,c)}(\theta)-\tilde{\gamma}_{cs^+(x,c)}(\theta-x)\right) d\theta  \\
& \ +  \frac{1}{2}r(c) \int_{x}^{2}  \left(\tilde{\gamma}_{cs^-(x,c)}(\theta)-\tilde{\gamma}_{cs^-(x,c)}(\theta-x)\right) d\theta\\ 
=& r(c) \int_{0}^x  \tilde{\gamma}_{cs^+(x,c)}(u) du  +  r(c) \int_{x}^{2}  \tilde{\gamma}_{cs^-(x,c)}(u)du\\ 
=& \ r(c) \int_0^x \left( \tilde{\gamma}_{cs^+(x,c)}(u)- \tilde{\gamma}_{cs^-(x,c)}(u) \right) du.
\end{align*} 
For the first equality we used the asymmetry of $\tilde{\gamma}_c$ and the change of variables $u=\theta-x$ on those terms involving $\theta-x$. For the second equality we used that $\tilde{\gamma}_c$ integrates to 0 on $(0,2)$. Using similar arguments,
\begin{align*}
a^c(x)=& \frac{1}{2} r(c) \int_{0}^x  \left(\tilde{\gamma}_{cs^+(x,c)}(\theta)-\tilde{\gamma}_{cs^+(x,c)}(\theta-x)\right)^2 d\theta \\
& \ +  \frac{1}{2}r(c) \int_{x}^{2}  \left(\tilde{\gamma}_{cs^-(x,c)}(\theta)-\tilde{\gamma}_{cs^-(x,c)}(\theta-x)\right)^2 d\theta\\ 
=& \ r(c) \left[ \int_0^x \left( \tilde{\gamma}^2_{cs^+(x,c)}(u)-\tilde{\gamma}^2_{cs^-(x,c)}(u)\right) du +2\int_0^1 \tilde{\gamma}^2_{cs^-(x,c)}(u) du  \right] \\
& \ -r(c) \int_0^x \left( \tilde{\gamma}_{cs^+(x,c)}(u) \tilde{\gamma}_{cs^+(x,c)}(u-x)-\tilde{\gamma}_{cs^-(x,c)}(u) \tilde{\gamma}_{cs^-(x,c)}(u-x) \right) du\\
& \ -r(c) \int_{-1}^1 \tilde{\gamma}_{cs^-(x,c)}(u) \tilde{\gamma}_{cs^-(x,c)}(u-x) du. 
\end{align*}
By suitable Taylor expansions of \eqref{eq:gammac}, it can be shown that 
\begin{equation} \label{eq:taylor}
\gamma_{c}(u)  = \begin{cases}  \sqrt{u^2+4c/\pi^2} (1+o(1)) \quad &\mbox{ if } \quad |u| \leq c^{1/2} \log c^{-1}; \\
 u+c\cot (\pi u/2)/\pi (1+o(1)) &\mbox{ if } \quad c^{1/2}\log c^{-1}< |u| < 1.
 \end{cases}
\end{equation}
Hence,  
\begin{align*}
b^c(x)=& \  c r(c) \big [ (s^+(x,c)-s^-(x,c)) \left(\log (c^{-1}) + 2\log \sin (\pi x/2)+1+2\log 2 \right)/\pi^2 \\
& \ -  \left ( s^+(x,c)\log (s^+(x,c))-s^-(x,c)\log (s^-(x,c) \right )/\pi^{2} \big ](1+o(1)) 
\end{align*}
and 
\begin{equation*}
a^c(x) = \frac{16 }{3 \pi^3} c^{3/2} r(c) \left (s^+(x,c)^{3/2}+s^-(x,c)^{3/2} \right )(1+o(1)).
\end{equation*}
Finally, we observe that the jumps in $(X_t^c)_{t \geq 0}$ are of order at most $c^{1/2}$. 
\end{proof}

Now recall that $s^+(x,c) \to s^+(x)$ and $s^-(x,c) \to s^-(x)$ uniformly on $x \in (0,2)$ as $c \to 0$, where $s^+$ and $s^-$ are Lipschitz continuous functions on $(0,2)$, not both identically zero. Setting $\tau = \inf\{t>0: X_t \notin (0,2) \}$, as above, we obtain the following.

\begin{corollary}\label{cor:maindiff}
\begin{itemize}
\item[(a)] Suppose that 
\(
r(c) = \left ( c \log (c^{-1}) \right )^{-1}.
\) 
If $(X_t)_{t \geq 0}$ is the solution to 
\begin{equation}\label{eq:odelim}
dX_t=(s^+(X_t)-s^-(X_t))dt
\end{equation}
starting from $X_0=1$, then $X_t^c \rightarrow X_{t \wedge \tau}$ weakly in $D[0,\infty)$.
\item[(b)] Suppose that 
\(
s^+(x)=s^-(x)=s(x)
\)
for all $x \in (0,2)$, that 
\(
r(c)=c^{-3/2},
\)
and that 
\[
c^{-1/2} \log (c^{-1}) (s^+(x,c)-s^-(x,c)) \rightarrow 0 
\]
uniformly in $x$ as $c \to 0$.
If $(X_t)_{t \geq 0}$ is the solution to 
\begin{equation*}
dX_t=\sqrt{\frac{32}{3 \pi^3}s(X_t)^{3/2}} dB_t
\end{equation*}
starting from $X_0=1$, then $X_t^c \rightarrow X_{t \wedge \tau}$ weakly in $D[0,\infty)$.
\item[(c)] Suppose that 
\(
s^+(x)=s^-(x)=s(x)
\)
 for all $x \in (0,2)$, that
 \(
 r(c)=c^{-3/2},
 \)
 and that 
\[
c^{-1/2} \log (c^{-1})(s^+(x,c)-s^-(x,c)) \rightarrow h(x) 
\]
uniformly in $x$ as $c \to 0$ for some Lipschitz continuous function $h$ on $(0,2)$.
If $(X_t)_{t \geq 0}$ is the solution to 
\begin{equation*}
dX_t=\frac{1}{\pi^2}h(X_t) dt + \sqrt{\frac{32}{3 \pi^3}s(X_t)^{3/2}} dB_t
\end{equation*}
starting from $X_0=1$, then $X_t^c \rightarrow X_{t \wedge \tau}$ weakly in $D[0,\infty)$.
\end{itemize}
\end{corollary}
\begin{remark}
Observe that taking $r(c)=c^{-3/2}$ and $s^+(x,c)=s^-(x,c)=1$, for all $x \in (0,2)$ and $c>0$, recovers the HL(0) result mentioned above.
\end{remark}

\section{Limit distributions}\label{sec:limit}
Our objective is to analyse the long-term behaviour of $X_t^c$, in the limit as $c \to 0$. Specifically, we would like to determine the distribution of
\[
X_{\infty} = \lim_{c \to 0} X_\infty^c = \lim_{c \to 0} \lim_{t \to \infty} X_t^c
\]
(if the limit exists), where all limits are in distribution. As we are interested in whether the red and blue regions of the associated growth model can coexist, we would like to know whether it is possible to select $s^\pm(x,c)$ in such a way that $X_{\infty}$ takes values on the interior of the interval $(0,2)$ with positive probability. Of particular interest is whether the distribution of $X_{\infty}$ can arise through non-trivial ergodic behaviour of the harmonic measure and take values on all of $(0,2)$.

Since rescaling time does not change $X_\infty^c$, we are free to take the rate of arrivals $r(c)$ as fast as possible whilst still ensuring that $X_t^c$ converges to a non-degenerate process $X_t$ as $c \to 0$. The strategy is to find conditions under which $\lim_{t \to \infty} X_{t\wedge \tau}$ has the required behaviour,  where $\tau = \inf\{t>0: X_t \notin (0,2) \}$, and then argue that under these conditions $X_{\infty} = \lim_{t \to \infty} X_{t \wedge\tau}$. Note that if $\tau<\infty$ then, using that convergence in $D[0,\infty)$ implies uniform convergence on compacts, by the Moore-Osgood Theorem $X_{\infty}$ is supported on $\{0,2\}$. In order to have a limit which takes values on all of $(0,2)$ it is therefore necessary to find conditions under which $\mathbb{P}(\tau=\infty)>0$.  

By Corollary \ref{cor:maindiff}, $X_t$ is either the solution to a deterministic ODE or to a diffusion process. In Section \ref{sec:det} we consider the deterministic case; in Section \ref{sec:diff} we analyse the limiting behaviour of diffusion processes.

\subsection{Deterministic limit process} \label{sec:det}
Firstly suppose that $s^+(x) \neq s^-(x)$ for all $x \in (0,2)$. Then taking $r(c)=(c \log(c^{-1}))^{-1}$, we obtain that $X_t^c \to X_t$ as $c \to 0$, where $X_t$ is the solution of the non-trivial ordinary differential equation \eqref{eq:odelim} as $c \to 0$. It is straightforward to deduce that if $s^+(x) > s^-(x)$ for all $x \in (0,2)$, then $\tau <\infty$ and $X_\tau=2$. Hence $X_{\infty}=2$ a.s.,~while if $s^+(x) < s^-(x)$ for all $x$, then $X_{\infty}=0$ a.s.

It is possible to choose $s^\pm(x)$ so that the ODE \eqref{eq:odelim} has stable fixed points in the interior of the interval $(0,2)$, in which case 
\(
\mathbb{P} \left (  X_\infty \in ( 0,2 ) \right ) >0.
\)
However the limit distribution will still be supported on a finite number of point masses.
For example, if $s^+(x)=2-x$ and $s^-(x)=x$ then $X_\infty^c \to 1$ in probability as $c \to 0$.
Showing this requires some care as the deterministic limit process does not in general converge to its fixed point in finite time. However, as we are more interested in whether co-existence can arise through non-trivial ergodic behaviour of the harmonic measure, we do not discuss the deterministic limit case in any further detail.

\subsection{Diffusion limit process} \label{sec:diff}
Now suppose that $s^+(x) = s^-(x)=s(x)>0$ for all $x \in (0,2)$. Then taking $r(c)=c^{-3/2}$, we obtain that $X_t^c \to X_t$ as $c \to 0$, where $X_t$ is the solution to an SDE of the form \eqref{eq:diff}
with $X_0=1$, for $a$ and $b$ as in Corollary \ref{cor:maindiff}.

\begin{definition}[Scale and speed]\label{def:scalespeed} For a diffusion \eqref{eq:diff}, define the scale function $\rho:(0,2) \to \mathbb{R}$ by 
\begin{equation*}
\rho(x)=\int_1^x \exp \left(-2 \int_1^y \frac{b(u)}{a(u)}du \right) dy
\end{equation*}
and the speed density, $m:(0,2) \to [0,\infty)$, by
\begin{equation*}
m(x)=\frac{1}{a(x)} \exp \left(2 \int_1^x \frac{b(y)}{a(y)}dy\right). 
\end{equation*}
\end{definition}

The existence of both is guaranteed provided that $b/a$ is locally integrable which is an assumption we shall make throughout. See Chapter 23 of \cite{KAL02} for further details about the scale function and speed measure, some of which we summarise below.

Suppose that $0<a<1<b<2$ and let $\tau_a = \inf\{t>0:X_t \leq a\}$ and $\tau_b = \inf\{t>0:X_t \geq b\}$. 
The scale function has the property that
\[
\mathbb{P}(\tau_a < \tau_b) = \frac{\rho(b)-\rho(1)}{\rho(b)-\rho(a)} = \frac{\rho(b)}{\rho(b)-\rho(a)}.
\]
Hence if $\rho(0):= \lim_{a \to 0^+} \rho(a) > - \infty$ and $\rho(2):= \lim_{b \to 2^-} \rho(b) < \infty$, then letting $a \to 0$ and $b \to 2$ in the expression above gives $\tau<\infty$ and
\[
\mathbb{P}(X_{\infty}=0)=1-\mathbb{P}(X_{\infty}=2)=\frac{\rho(2)}{\rho(2)-\rho(0)}.
\]
Observe that when $b(x)=0$ for all $x$, $\rho(x)=x-1$ and hence
\[
\mathbb{P}(X_{\infty}=0)=1-\mathbb{P}(X_{\infty}=2)=\frac{1}{2}.
\]
Similarly, it can be shown that if $\rho(0)=-\infty$ and $\rho(2)<\infty$, then
\[
\mathbb{P}(X_{\infty}=0)=1-\mathbb{P}(X_{\infty}=2)=0
\]
and if $\rho(0)>-\infty$ and $\rho(2)=\infty$, then
\[
\mathbb{P}(X_{\infty}=0)=1-\mathbb{P}(X_{\infty}=2)=1.
\]

Finally, we consider the case when $\rho(0)=-\infty$ and $\rho(2)=\infty$. In this case $X_t$ is recurrent if the speed density $m$ is integrable on $(0,2)$; otherwise $X_t$ is null-recurrent. In the case that $X_t$ is recurrent, $\tau = \infty$ and $X_t$ has an asymptotic distribution which has density proportional to $m$. This situation can only arise if we are in case (c) of Corollary \ref{cor:maindiff}. 

\begin{theorem} \label{thrm:mainerg}
Suppose that 
\(
s^+(x)=s^-(x)=s(x)>0
\)
 for all $x \in (0,2)$, that
 \(
 r(c) =  c^{-3/2}
 \)
 and that 
\[
c^{-1/2} \log (c^{-1}) (s^+(x,c)-s^-(x,c)) \rightarrow h(x) 
\]
uniformly in $x$ as $c \to 0$ for some Lipschitz continuous function $h$ on $(0,2)$ for which $h(x)/s(x)^{3/2}$ is locally integrable on $(0,2)$.
Setting 
\[
a(x)=32 s(x)^{3/2}/(3 \pi^3) \quad \mbox{ and } \quad b(x)=h(x)/\pi^2,
\]
define the scale function $\rho$ and speed measure $m$ as in Definition \ref{def:scalespeed}.
Suppose that the following conditions hold:
\begin{enumerate}
\item[(i)]
$\rho(x) \to -\infty$ as $x \to 0$ and $\rho(x) \to \infty$ as $x \to 2$;
\item[(ii)]
\[
M:=\int_0^2 m(x) dx < \infty;
\] 
\item[(iii)] There exist constants $C > D>0$ such that, for all $x \in (0,2)$,
\begin{equation} \label{eq:lyapquad}
2(x-1)b(x) + a(x)  \leq  - C(x-1)^2 + D. 
\end{equation}
\end{enumerate}
Then $X_{\infty}=\lim_{c \to 0} \lim_{t \to \infty} X_t^c$ exists and has density $m(x)/M$ on $(0,2)$. 
\end{theorem}
\begin{proof}
By Theorem 23.15 of \cite{KAL02}, conditions (i) and (ii) guarantee that for any solution $X_t$ to the SDE \eqref{eq:diff}, the weak limit $\lim_{t \to \infty} X_t$ exists, regardless of the initial distribution $X_0$, and has density $m(x)/M$. We denote this distribution by $\pi_{\infty}$. 

We now show that if $c$ is sufficiently small, then there is an invariant distribution for the process $X_t^c$.
By the proof of Proposition \ref{thrm:maindiff}, $a^c \to a$ and $b^c \to b$ uniformly as $c \to 0$. Therefore, there exists some $c_0>0$ such that
\[
\|a-a^c\|_{\infty}+2\|b-b^c\|_{\infty} < (C-D)/2
\] 
for all $c<c_0$. In everything that follows, assume that $c<c_0$.
Let $L^c$ be the generator corresponding to $X_t^c$. Then, taking $V(x)=(x-1)^2+1$, by \eqref{eq:lyapquad},
\[
L^c V(x) = a^c(x)+2b^c(x) \leq -C(x-1)^2+(C+D)/2.
\]
The function $V(x)$ is therefore a {\em Lyapounov function} for the process $X_t^c$ satisfying 
\begin{equation}
\label{eq:lyap}
L^cV(x) \leq -C'V(x) + D'I_{K}(x),
\end{equation}
with $K=\{x:|x-1| \leq \eta \}$ for any $\sqrt{(C+D)/(2C)} < \eta < 1$, and $C', D' >0$ (which depend on the choice of $\eta$). By Theorem 4.5 of \cite{MEY93}, $X_t^c$ has an invariant distribution $\pi^c$.

Since the space of probability measures on $[0,2]$ is compact, every subsequence of $\pi^c$ has a convergent subsequence. Suppose $\pi^{c_n} \to \pi^0$ is one such subsequence. Let $X^n_t$ be the pure jump process with generator $L^{c_n}$, starting from $X^n_0 \sim \pi^{c_n}$ and let $X^0_t$ be the solution to the SDE \eqref{eq:diff}, starting from  $X_0 \sim \pi^0$. By Theorem \ref{thrm:kurtz}, $X^n_t \to X^0_t$ weakly in $D[0, \infty)$. Since $X^n_0$ was started in its invariant distribution, $X^n_t \sim \pi^{c_n}$ for all $t$ and hence $X^0_t \sim \pi^0$ for all $t$. However, $\lim_{t \to \infty} X^0_t \sim \pi_{\infty}$ and therefore $\pi^0 = \pi_{\infty}$. Since this holds for every convergent subsequence of $\pi^c$, it follows that $\pi^c \to \pi_{\infty}$ as $c \to 0$.

It is therefore enough to show that for $c$ sufficiently small $\lim_{t \to \infty} X_t^c \sim \pi^c$. Suppose that $\bar{X}^c_t$ is the pure jump process with generator $L^{c}$, starting from $\bar{X}^c_0 \sim \pi^{c}$, and coupled with $X^c_t$ so that if $T$ is the first time that $X^c_T = \bar{X}^c_T$, then $X^c_t = \bar{X}^c_t$ for all $t \geq T$. Since $\bar{X}^c_t \sim \pi^c$ for all $t$, $X^c_t \sim \pi^c$ for all $t \geq T$ and therefore $\lim_{t \to \infty} X_t^c \sim \pi^c$ as required. It is therefore sufficient to construct a coupling in which $T < \infty$ a.s.

By Corollary 2.9 in \cite{LOC17}, \eqref{eq:lyap} implies that for any  coupling, the two processes $X^{c}_t$ and $\bar{X}^{c}_t$ visit the compact set $K$ at the same time infinitely often, almost surely. 
For each $x \in K$, let $X^{c,x}_t$ denote the process with generator $L^{c}$, starting from $X^{c,x}_0=x$. 
Suppose we can couple the processes $X^{c,x}_t$ and $X^{c,y}_t$ so that after a unit of time
\begin{equation}\label{eq:posprob}
\inf_{x,y \in K}\mathbb{P}(X^{c,x}_{1} = X^{c,y}_{1}) > 0.
\end{equation}
Then every time the processes $X^{c}_t$ and $\bar{X}^{c}_t$ are simultaneously in $K$, there is a positive chance that they will have coalesced one unit of time later. It follows that
\[
\mathbb{P}(T=\infty) \leq \lim_{n \to \infty} \left (1-\inf_{x,y \in K}\mathbb{P}(X^{c,x}_{1} = X^{c,y}_{1})\right)^n = 0.
\]
By using \eqref{eq:kernel}, \eqref{eq:taylor}, that $s^\pm(x,c) \to s(x)$ uniformly on $(0,2)$, and that, by the compactness of $K$, $s(x)$ is uniformly bounded away from $0$ and $\infty$ on $K$, it can be shown that there exist $a^0, b^0, c^0 \in (0,1)$ such that for all $c<c^0$, and measurable $A \subseteq [-b^0 \sqrt{c}, -b^0 \sqrt{c}/8] \cup [b^0 \sqrt{c}/8, b^0 \sqrt{c}]$,
\[
\inf_{x \in K} \mathbb{P}(\Delta X^c_{t_n} \in A | X^c_{t_{n-1}}=x) \geq a^0 |A|,
\]
where $|A|$ denotes the Lebesgue measure of $A$ and $t_n$ is the $n^{th}$ jump time of $X_t^{c}$.
Let $K^0=\{x: |x-1| \leq b^0\sqrt{c}/8 \}$ and let $\nu$ denote the uniform probability measure  on $A_0=\{x: b^0\sqrt{c}/4 < |x-1| < 7b^0\sqrt{c}/8 \}$. Then if $\beta=a_0|A_0|$, $0<\beta < 1$ and 
\[
\inf_{x \in K^0} \mathbb{P}(X^{c}_{t_n} \in A | X^{c}_{t_{n-1}}=x) \geq a_0 |A \cap A_0| = \beta \nu(A)
\]
for all measurable sets $A \subseteq (0,2)$, where we have used translation invariance of Lebesgue measure here. 

We now describe the coupling for $X_t^{c,x}$ and $X_t^{c,y}$ with $x,y \in K$. Let the processes evolve independently until the first time that they are both in the set $K^0$. At this point sample the time until the next jump to be the same Exp$(r(c))$ random variable for both processes (which can be done by the memoryless property of the exponential distribution). Then, with probability $\beta$, sample the new position for both processes from $\nu$, from which point onwards the processes remain equal; otherwise sample the new positions independently from the respective distributions $A \mapsto (\mathbb{P}(\Delta X^c_{t_n} \in A | X^c_{t_{n-1}}=x)-\beta \nu (A))/(1-\beta)$ and $A \mapsto (\mathbb{P}(\Delta X^c_{t_n} \in A | X^c_{t_{n-1}}=y)-\beta \nu (A))/(1-\beta)$ and allow the processes to evolve independently until they are next both in $K^0$. At this point, attempt to coalesce again as above, and repeat. 

Finally we need to show that \eqref{eq:posprob} holds for processes coupled in this way. For simplicity, suppose that $1-\eta \leq x,y \leq 1$; a similar argument works in the other cases. Set $A_1=[b^0 \sqrt{c}/8,b^0 \sqrt{c}/4]$,  $N=\lfloor 8/(b^0 \sqrt{c})\rfloor + 1$, and $s_k=k/N$ for $k=0, \dots, N$.  Let $t_k^x$ and $t_k^y$ be the respective jump times of $X_t^{c,x}$ and $X_t^{c,y}$, set $N^x = \inf\{ k: X_{t_k^x}^{c,x} \in K_0\}$ and define $N^y$ similarly. Define $A_{x,y}$ to be the intersection of the three events
$\{ s_{k-1} \leq t_k^x < s_{k}, \Delta X^{c,x}_{t^x_{k}} \in A_1, k=1, \dots, N^x \}$, $\{ s_{k-1} \leq t_k^y < s_{k}, \Delta X^{c,y}_{t^y_{k}} \in A_1, k=1, \dots, N^y \}$ and $\{ s_{N-1} \leq t^x_{N^x+1} = t^y_{N^y+1} < s_N, X^{c,x}_{t^x_{N^x+1}}=X^{c,y}_{t^y_{N^y+1}}\}$. Hence
\[
\inf_{x,y \in K}\mathbb{P}(X^{c,x}_{1} = X^{c,y}_{1}) \geq \inf_{x,y \in K}\mathbb{P}(A_{x,y})\geq e^{-2r(c)}(a_0|A_1|r(c)/N)^{2N} \beta >0, 
\]
as required.
\end{proof}

\begin{remark}
The condition \eqref{eq:lyapquad} can be relaxed to the weaker assumption requiring the existence of a Lyapounov function and compact set $K$ that satisfies \eqref{eq:lyap} with generator $L^c$, for every $c>0$. Although the current condition is slightly more restrictive, it is straightforward to verify for specific examples.
\end{remark}

\section{An example of coexistence}\label{sec:example}

The above analysis identifies sufficient conditions for ensuring coexistence between the two regions. It shows that constructing cases  in which the limiting distribution arises through non-trivial ergodic behaviour of the harmonic measure is very delicate. The way in which the sizes of particles must be tuned depending on which region they land in, and how close to the boundary they land, is subtle. In this section we give an example of a process in which the necessary conditions are satisfied.

\begin{figure}[h]
\begin{center}
\includegraphics[scale=0.5]{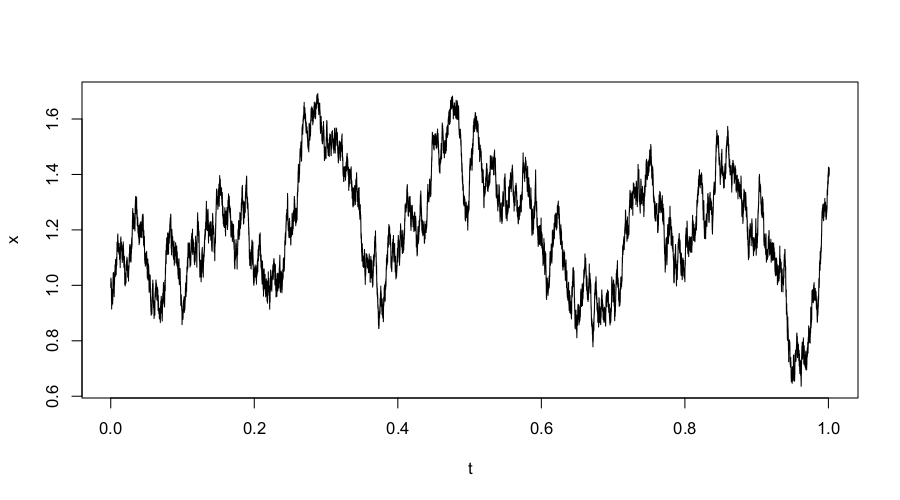} 
\includegraphics[scale=0.55]{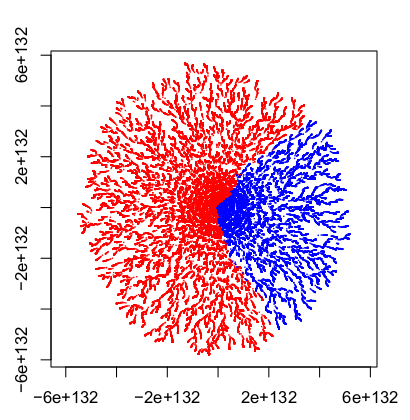} 
\includegraphics[scale=0.5]{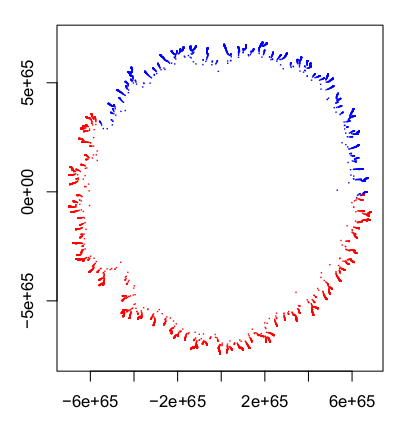} 
\caption{Simulation of the model described in Section \ref{sec:example} with $c=10^{-4}$. The top image is a realisation of the process $X_t^c$. The bottom left image shows the cluster corresponding to this realisation at time $t=1$. Due to the distortion of particle shapes of the HL(0) model mentioned in Section \ref{sec:harm}, almost all of the visible cluster comprises particles that arrived in the final 1\% of the time. The bottom right image gives a snapshot of where the harmonic measure is concentrated in each region of the cluster at time $t=0.5$. (Note the change in scale on the axes between the two images).} \label{fig:example} 
\end{center}
\end{figure}

Set 
\begin{align*}
s^+(x,c)&=\pi^2 \left ( \left (\frac{3x(2-x)}{16} \right )^{2/3} + c^{1/2} (\log(c^{-1}))^{-1}(2-x) \right ), \\
s^-(x,c)&=\pi^2 \left (  \left (\frac{3x(2-x)}{16} \right )^{2/3} + c^{1/2} (\log(c^{-1}))^{-1}x \right ).
\end{align*}
Taking $r(c)=c^{-3/2}$, the conditions of Corollary \ref{cor:maindiff} (c) are satisfied with  $s(x)=\pi^2  (3x(2-x)/16)^{2/3}$ and $h(x)=2\pi^2(1-x)$. Hence $X_t^c \to X_t$ weakly in $D[0,\infty)$ as $c \to 0$, where
\[
dX_t= 2(1-X_t)dt+\sqrt{2 X_t(2-X_t)} dB_t.
\]
The scale function and speed density are given by
\begin{align*}
\rho(x) &= \int_1^x \exp \left ( -2 \int_1^y \frac{1-u}{u(2-u)} du \right  ) = \int_1^x \frac{1}{y(2-y)} dy = \log \sqrt{\frac{x}{2-x}},
\\
m(x) &= \frac{1}{2x(2-x)}\exp \left ( 2 \int_1^x \frac{1-u}{u(2-u)} du \right  ) = \frac{1}{2x(2-x)} \exp \left ( \log x + \log (2-x) \right )  = \frac{1}{2}.
\end{align*}
It is straightforward to check that the conditions of Theorem \ref{thrm:mainerg} are satisfied. It follows that the weak limit $X_{\infty}$ exists and is uniformly distributed on $(0,2)$.

A simulation of the process $(X_t^c)_{t \geq 0}$ and corresponding cluster is shown in Figure \ref{fig:example}.


\bibliographystyle{amsplain}


\providecommand{\bysame}{\leavevmode\hbox to3em{\hrulefill}\thinspace}
\providecommand{\MR}{\relax\ifhmode\unskip\space\fi MR }
\providecommand{\MRhref}[2]{%
  \href{http://www.ams.org/mathscinet-getitem?mr=#1}{#2}
}
\providecommand{\href}[2]{#2}


\ACKNO{We are grateful to Eva L\"ocherbach for directing us to the paper \cite{LOC17}, which provided the key ideas for the proof of Theorem \ref{thrm:mainerg}, and to the anonymous referee whose suggestions significantly improved the paper.
}


\end{document}